\documentclass{llncs}
\usepackage{amsfonts}
\usepackage{amsmath}
\usepackage{tikz}

\begin{document}

\title{Cycle structure in SR and DSR graphs: implications for multiple equilibria and stable oscillation in chemical reaction networks}

\author{Murad Banaji}

\institute{Department of Medical Physics and Bioengineering, University College London, Gower Street, London WC1E 6BT, UK.}

\maketitle

\begin{abstract}
Associated with a chemical reaction network is a natural labelled bipartite multigraph termed an SR graph, and its directed version, the DSR graph. These objects are closely related to Petri nets. The construction of SR and DSR graphs for chemical reaction networks is presented. Conclusions about asymptotic behaviour of the associated dynamical systems which can be drawn easily from the graphs are discussed. In particular, theorems on ruling out the possibility of multiple equilibria or stable oscillation in chemical reaction networks based on computations on SR/DSR graphs are presented. These include both published and new results. The power and limitations of such results are illustrated via several examples. 
\end{abstract}

\section{Chemical reaction networks: structure and kinetics}
\label{secintro}

Models of chemical reaction networks (CRNs) are able to display a rich variety of dynamical behaviours \cite{pojman}. In this paper, a spatially homogeneous setting is assumed, so that CRNs involving $n$ chemicals give rise to local semiflows on $\mathbb{R}^n_{\geq 0}$, the nonnegative orthant in $\mathbb{R}^n$. These local semiflows are fully determined if we know 1) the CRN {\em structure}, that is, which chemicals react with each other and in what proportions, and 2) the CRN {\em kinetics}, that is, how the reaction rates depend on the chemical concentrations. An important question is what CRN behaviours are determined primarily by reaction network structure, with limited assumptions about the kinetics. 

A variety of representations of CRN structure are possible, for example via matrices or generalised graphs. Of these, a signed, labelled, bipartite multigraph, termed an {\em SR graph}, and its directed version, the {\em DSR graph}, are formally similar to Petri nets. This relationship is discussed further below. 

It is now well established that graphical representations can tell us a great deal about asymptotic behaviours in the associated dynamical systems. Pioneering early work on CRNs with mass-action kinetics (\cite{feinberg0,feinberg} for example), had a graph-theoretic component (using graphs somewhat different from those to be presented here). More recently, graph-theoretic approaches have been used to draw conclusions about multistationarity and oscillation in CRNs with restricted classes of kinetics \cite{craciun1,minchevaroussel}. 

The applicability of such work, particularly in biological contexts, is greatly increased if only weak assumptions are made about kinetics. Consequently, there is a growing body of recent work on CRNs with essentially arbitrary kinetics. It has been shown that examination of Petri nets associated with a CRN allows conclusions about persistence, that is, whether $\omega$-limit sets of interior points of $\mathbb{R}^n_{\geq 0}$ can intersect the boundary of $\mathbb{R}^n_{\geq 0}$ \cite{angelipetrinet}. Work on multistationarity has been extended beyond the mass-action setting \cite{banajicraciun,banajicraciun2}: some conclusions of this work will be outlined below. Finally, recent work applying the theory of monotone dynamical systems \cite{hirschsmith,halsmith} in innovative ways to CRNs \cite{angelileenheersontag} has close links with some of the new material presented below.

{\bf Outline}. After some preliminaries, the construction of SR and DSR graphs is presented, and their relationship to Petri nets is discussed. Some recent results about multistationarity based on cycle structure in these objects are described. Subsequently, a new result on monotonicity in CRNs is proved. This result, Proposition~\ref{mainthm}, is a graph-theoretic corollary of results in \cite{banajidynsys}. It bears an interesting relationship to results in \cite{angelileenheersontag}, which provide stronger conclusions about convergence, but make different assumptions, and a somewhat different claim. Finally, several examples, some raising interesting open questions, are presented. At various points, in order to simplify the exposition, the results are presented in less generality than possible, with more technical results being referenced. 

\section{Preliminaries}

\subsection{A motivating example}
Consider the following simple family of CRNs treated in \cite{craciun,banajiSIAM}:
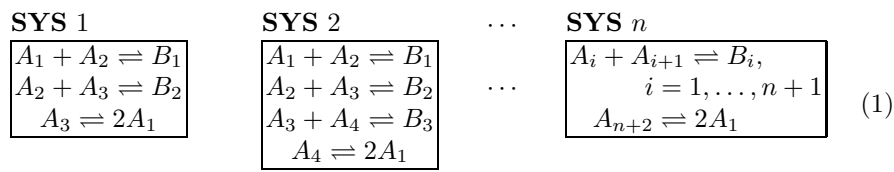
\begin{equation}
\label{eqfamily}
\begin{array}{cccccccccc}
\begin{array}{l}\mathbf{SYS\,\,}1\\\begin{array}{|c|}\hline A_1 + A_2 \rightleftharpoons B_1\\
A_2 + A_3 \rightleftharpoons B_2\\
A_3 \rightleftharpoons 2A_1\\\hline
\end{array}
\\{}\\{}
\end{array} &\qquad &
\begin{array}{l}\mathbf{SYS\,\,}2\\\begin{array}{|c|}\hline A_1 + A_2 \rightleftharpoons B_1\\
A_2 + A_3 \rightleftharpoons B_2\\
A_3 +A_4 \rightleftharpoons B_3\\
A_4 \rightleftharpoons 2A_1\\\hline
\end{array}
\\{}
\end{array} & \quad &
\begin{array}{l}\cdots\\{}\\\cdots\\{}\\{}\\{}
\end{array} & \quad &
\begin{array}{l}\mathbf{SYS\,\,}n\\\begin{array}{|c|}\hline \hspace{-0.8cm} A_i + A_{i+1} \rightleftharpoons B_i,\\{\hspace{1cm} i = 1, \ldots, n+1}\\\hspace{-0.8cm} A_{n+2}  \rightleftharpoons 2A_1 \\\hline
\end{array}
\\{}\\{}
\end{array}
\end{array}
\end{equation}
The reader may wish to look ahead to Figure~\ref{SRsequence} to see representations of the SR graphs associated with the first three CRNs in this family. This family will be revisited in Section~\ref{secexamples}, and the theory to be presented will imply the following conclusions (to be made precise below): when $n$ is even, {\bf SYS}~$n$ does not allow multiple nondegenerate equilibria; when $n$ is odd, {\bf SYS}~$n$ cannot have a nontrivial periodic attractor. Both conclusions require only minimal assumptions about the kinetics. 

\subsection{Dynamical systems associated with CRNs}

In a spatially homogeneous setting, a chemical reaction system in which $n$ reactants participate in $m$ reactions has dynamics governed by the ordinary differential equation
\begin{equation}
\dot x = \Gamma v(x). \label{basic}
\end{equation}
$x = [x_1, \ldots, x_n]^T$ is the nonnegative vector of reactant concentrations, and $v = [v_1, \ldots, v_m]^T$ is the vector of reaction rates, assumed to be $C^1$. A reaction rate is the rate at which a reaction proceeds {\em to the right} and may take any real value. $\Gamma$ is the (constant) $n \times m$ {\bf stoichiometric matrix} of the reaction system. Since reactant concentrations cannot be negative, it is always reasonable to assume invariance of $\mathbb{R}^n_{\geq 0}$, i.e. $x_i = 0 \Rightarrow \dot x_i \geq 0$. 

The $j$th column of $\Gamma$, termed $\Gamma_j$, is the {\bf reaction vector} for the $j$th reaction, and a stoichiometric matrix is defined only up to an arbitrary signing of its columns. In other words, given any $m \times m$ signature matrix $D$ (i.e. any diagonal matrix with diagonal entries $\pm 1$), one could replace $\Gamma$ with $\Gamma D$ and $v(x)$ with $Dv(x)$. Obviously the dynamical system is left unchanged. The subspace $\mathrm{Im}(\Gamma)$ of $\mathbb{R}^n$ spanned by the reaction vectors is called the {\bf stoichiometric subspace}. The intersection of any coset of the $\mathrm{Im}(\Gamma)$ with $\mathbb{R}^n_{\geq 0}$ is called a {\bf stoichiometry  class}.

Two generalisations of (\ref{basic}) which include explicit inflow and outflow of substrates are worth considering. The first of these is a so-called CFSTR
\begin{equation}
\dot x = q(x_{in} - x) + \Gamma v(x). \label{basicCFSTR}
\end{equation}
$q \in \mathbb{R}$, the flow rate, is generally assumed to be positive, but we allow $q=0$ so that (\ref{basic}) becomes a special case of (\ref{basicCFSTR}). $x_{in} \in \mathbb{R}^n$ is a constant nonnegative vector representing the ``feed'' (i.e., inflow) concentrations. The second class of systems is:
\begin{equation}
\dot x = x_{in} + \Gamma v(x) - Q(x). \label{basicoutflows}
\end{equation}
Here $Q(x) = [q_1(x_1), \ldots, q_n(x_n)]^T$, with each $q_i(x_i)$ assumed to be a $C^1$ function satisfying $\frac{\partial q_i}{\partial x_i} > 0$, and all other quantities defined as before. Systems (\ref{basicoutflows}) include systems (\ref{basicCFSTR}) with $q \not = 0$, while systems (\ref{basic}) lie in the closure of systems (\ref{basicoutflows}).

Define the $m \times n$ matrix $V = [V_{ji}]$ where $V_{ji} = \frac{\partial v_j}{\partial x_i}$. A very reasonable, but weak, assumption about many reaction systems is that reaction rates are monotonic functions of substrate concentrations as assumed in \cite{banajiSIAM,leenheer,banajimathchem} amongst other places. We use the following definition from \cite{banajiSIAM} (there called NAC):

\begin{quote}
A reaction system is {\bf N1C} if i) $\Gamma_{ij}V_{ji} \leq 0$ for all $i$ and $j$, and ii) $\Gamma_{ij} = 0 \Rightarrow V_{ji} = 0$. 
\end{quote}

As discussed in \cite{banajiSIAM}, the relationship between signs of entries in $\Gamma$ and $V$ encoded in the N1C criterion is fulfilled by all reasonable reaction kinetics (including mass action and Michaelis-Menten kinetics for example), provided that reactants never occur on both sides of a reaction.

\section{Introduction to SR and DSR graphs}

\subsection{Construction and relation to Petri nets}

SR graphs are signed, bipartite multigraphs with two vertex sets $V_S$ (termed ``S-vertices'') and $V_R$ (termed ``R-vertices''). The edges $E$ form a multiset, consisting of unordered pairs of vertices, one from $V_S$ and one from $V_R$. Each edge is signed and labelled either with a positive real number or the formal label $\infty$. In other words, there are functions $\mathrm{sgn}:E \to \{-1, 1\}$, and $\mathrm{lbl}:E \to (0, \infty)\cup \{\infty\}$. The quintuple $(V_S, V_R, E, \mathrm{sgn}, \mathrm{lbl})$ defines an SR graph. 

DSR graphs are similar, but have an additional ``orientation function'' on their edges, $\mathcal{O}: E \to \{-1, 0, 1\}$. The sextuple $(V_S, V_R, E, \mathrm{sgn}, \mathrm{lbl}, \mathcal{O})$ defines a DSR graph. If $\mathcal{O}(e) = -1$ we will say that the edge $e$ has ``S-to-R direction'', if $\mathcal{O}(e) = 1$, then $e$ has ``R-to-S direction'', and if $\mathcal{O}(e) = 0$, then $e$ is ``undirected''. An undirected edge can be regarded as an edge with both S-to-R and R-to-S direction, and indeed, several results below are unchanged if an undirected edge is treated as a pair of antiparallel edges of the same sign. SR graphs can be regarded as the subset of DSR graphs where all edges are undirected. 

Both the underlying meanings, and the formal structures, of Petri nets and SR/DSR graphs have some similarity. If we replace each undirected edge in a DSR graph with a pair of antiparallel edges, a DSR graph is simply a Petri net graph, i.e. a bipartite, multidigraph. Similarly, an SR graph is a bipartite multigraph. S-vertices correspond to variables, while R-vertices correspond to processes which govern their interaction. The notions of variable and process are similar to the notions of ``place'' and ``transition'' for a Petri net. Edges in SR/DSR graphs tell us which variables participate in each process, with additional qualitative information on the nature of this participation in the form of signs, labels, and directions; edges in Petri nets inform on which objects are changed by a transition, again with additional information in the form of labels (multiplicities) and directions. Thus both Petri net graphs and SR/DSR graphs encode partial information about associated dynamical systems, while neither includes an explicit notion of time. 

There are some important differences, however. Where SR/DSR graphs generally represent the structures of continuous-state, continuous-time dynamical systems, Petri nets most often correspond to discrete-state, discrete-time systems, although the translation to a continuous-state and continuous-time context is possible \cite{DavidAlla}. Although in both cases additional structures give partial information about these dynamical systems, there are differences of meaning and emphasis. Signs on edges in a DSR graph, crucial to much of the associated theory, are analogous to directions on edges in a Petri net: for example for an irreversible chemical reaction, an arc from a substrate to reaction vertex in the Petri net would correspond to a negative, undirected, edge in the SR/DSR graph. Unlike SR/DSR graphs, markings (i.e. vertex-labellings representing the current state) are often considered an intrinsic component of Petri nets. 

Apart from formal variations between Petri nets and SR/DSR graphs, differences in the notions of state and time lead naturally to differences in the questions asked. Most current work using SR/DSR graphs aims to inform on the existence, nature, and stability of limit sets of the associated local semiflows. Analogous questions are certainly possible with Petri nets, for example questions about the existence of stationary probability distributions for stochastic Petri nets \cite{bause}. However, much study, for example about reachability, safeness and boundedness, concerns the structure of the state space itself, and has no obvious analogy in the SR/DSR case. This explains to some extent the importance of markings in the study of Petri nets; in the case of SR/DSR graphs, the underlying space is generally assumed to have a simple structure, and the aim is to draw conclusions which are largely independent of initial conditions.

\subsection{SR and DSR graphs associated with CRNs}

SR and DSR graphs can be associated with arbitrary CRNs and more general dynamical systems \cite{banajicraciun,banajicraciun2}. For example, the construction extends to situations where there are modulators of reactions which do not themselves participate in reactions, and where substrates occur on both sides of a reaction. Here, for simplicity, the construction is presented for an N1C reaction system with stoichiometric matrix $\Gamma$. Assume that there is a set of substrates $V_S = \{S_1, \ldots, S_n\}$, having concentrations $x_1, \ldots, x_n$, and reactions $V_R = \{R_1, \ldots, R_m\}$ occurring at rates $v_1, \ldots, v_m$. The labels in $V_S$ and $V_R$ will be used to refer both to the substrate/reaction, and the associated substrate/reaction vertices. 
\begin{itemize}
\item If $\Gamma_{ij} \not = 0$ (i.e. there is net production or consumption of $S_i$ reaction $j$), and also $\frac{\partial v_j}{\partial x_i}$ is not identically zero, i.e. the concentration of substrate $i$ affects the rate of reaction $j$, then there is an undirected edge $\{S_i, R_j\}$. 
\item If $\Gamma_{ij} \not = 0$, but $\frac{\partial v_j}{\partial x_i} \equiv 0$, then the edge $\{S_i, R_j\}$ has only R-to-S direction.
\end{itemize}
The edge $\{S_i, R_j\}$ has the sign of $\Gamma_{ij}$ and label $|\Gamma_{ij}|$. Thus the labels on edges are just stoichiometries, while the signs on edges encode information on which substrates occur together on each side of a reaction. A more complete discussion of the meanings of edge-signs in terms of ``activation'' and ``inhibition'' is presented in \cite{banajicraciun2}. Note that in the context of N1C reaction systems, the following features (which are reasonably common in the more general setting) do not occur: edges with only R-to-S direction; multiple edges between a vertex pair; and edges with edge-label $\infty$. 

SR/DSR graphs can be uniquely associated with (\ref{basic}), (\ref{basicCFSTR}), or (\ref{basicoutflows}): in the case of (\ref{basicCFSTR})~and~(\ref{basicoutflows}), the inflows and outflows are ignored, and the SR/DSR graph is just that derived from the associated system (\ref{basic}). The construction is most easily visualised via an example. Consider, first, the simple system of two reactions: 
\begin{equation}
A + B \rightleftharpoons C, \qquad A \rightleftharpoons B \label{eqSR1}
\end{equation}
This has SR graph, shown in Figure~\ref{figSR1},~{\em left}. If all substrates affect the rates of reactions in which they participate then this is also the DSR graph for the reaction. If, now, the second reaction is irreversible, i.e. one can write
\begin{equation}
A + B \rightleftharpoons C, \qquad A \rightarrow B, \label{eqSR2}
\end{equation}
and consequently the concentration of $B$ does not affect the rate of the second reaction\footnote{Note that this is usually, but not always, implied by irreversibility: it is possible for the product of an irreversible reaction to influence a reaction rate.}, then the SR graph remains the same, losing information about irreversibility, but the DSR graph now appears as in Figure~\ref{figSR1}~{\em right}.

\begin{figure}[h]
\begin{center}
\begin{minipage}{0.3\textwidth}
\begin{tikzpicture}[domain=0:4,scale=0.6]
    
\node[draw,shape=circle] at (1,4) {$A$};
\node[draw,shape=rectangle] at (4,4) {$R_1$};
\draw[-, dashed, thick] (1.8,4) -- (3.2,4);
\node[draw,shape=circle] at (2.5,2.5) {$C$};
\draw[-, thick] (3,3) -- (3.5, 3.5);
\node[draw,shape=circle] at (4,1) {$B$};
\node[draw,shape=rectangle] at (1,1) {$R_2$};
\draw[-, thick] (1.6,1) -- (3.2,1);
\draw[-, dashed, thick] (1,1.8) -- (1,3.2);
\draw[-, dashed, thick] (4,1.8) -- (4,3.2);

\node at (0.6,2.5) {$\scriptstyle{1}$};
\node at (4.4,2.5) {$\scriptstyle{1}$};
\node at (2.5,0.6) {$\scriptstyle{1}$};
\node at (2.5,4.4) {$\scriptstyle{1}$};
\node at (3.1,3.5) {$\scriptstyle{1}$};
\end{tikzpicture}
\end{minipage}
\hspace{2cm}
%\hfill
\begin{minipage}{0.3\textwidth}
\begin{tikzpicture}[domain=0:4,scale=0.6]
\node[draw,shape=circle] at (1,4) {$A$};
\node[draw,shape=rectangle] at (4,4) {$R_1$};
\draw[-, dashed, thick] (1.8,4) -- (3.2,4);
\node[draw,shape=circle] at (2.5,2.5) {$C$};
\draw[-, thick] (3,3) -- (3.5, 3.5);

\node[draw,shape=circle] at (4,1) {$B$};
\node[draw,shape=rectangle] at (1,1) {$R_2$};
\draw[->, thick] (1.6,1) -- (3.2,1);

\draw[-, dashed, thick] (1,1.8) -- (1,3.2);

\draw[-, dashed, thick] (4,1.8) -- (4,3.2);

\node at (0.6,2.5) {$\scriptstyle{1}$};
\node at (4.4,2.5) {$\scriptstyle{1}$};
\node at (2.5,0.6) {$\scriptstyle{1}$};
\node at (2.5,4.4) {$\scriptstyle{1}$};
\node at (3.1,3.5) {$\scriptstyle{1}$};
\end{tikzpicture}
\end{minipage}
\end{center}
\caption{\label{figSR1} {\em Left.} The SR (and DSR graph) for reaction system (\ref{eqSR1}). Negative edges are depicted as dashed lines, while positive edges are bold lines. This convention will be followed throughout. {\em Right.} The DSR graph for reaction system (\ref{eqSR2}), that is, when $B$ is assumed not to affect the rate of the second reaction.}
\end{figure}
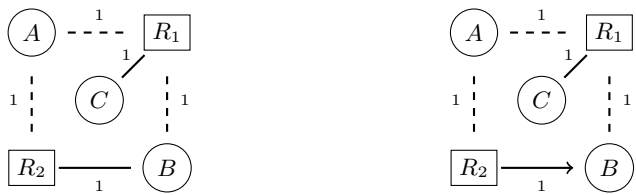

\section{Paths and cycles in SR and DSR graphs}

In the usual way, {\bf cycles} in SR (DSR) graphs are minimal undirected (directed) paths from some vertex to itself. All paths have a sign, defined as the product of signs of edges in the path. Given any subgraph $E$, its size (or length, if it is a path) $|E|$ is the number of edges in $E$. Paths of length two will be called {\bf short} paths. Any path $E$ of even length also has a {\bf parity}
\[
P(E) = (-1)^{|E|/2}\mathrm{sign}(E).
\]
A cycle $C$ is an {\bf e-cycle} if $P(C) = 1$, and an {\bf o-cycle} otherwise. Given a cycle $C$ containing edges $e_1, e_2, \ldots, e_{2r}$ such that $e_i$ and $e_{(i\bmod 2r) + 1}$ are adjacent for each $i =1,\ldots,2r$, define:
\[
\mathrm{stoich}(C) = \left|\prod_{i = 1}^{r}\mathrm{lbl}(e_{2i-1}) - \prod_{i = 1}^{r}\mathrm{lbl}(e_{2i})\right|\,.
\]
Note that this definition is independent of the starting point chosen on the cycle. A cycle with $\mathrm{stoich}(C) = 0$ is termed an {\bf s-cycle}. 

An {\bf S-to-R path} in an SR graph is a non-self-intersecting path between an S-vertex and an R-vertex. {\bf R-to-R paths} and {\bf S-to-S paths} are similarly defined, though in these cases the initial and terminal vertices may coincide. Any cycle is both an R-to-R path and an S-to-S path. Two cycles have {\bf S-to-R intersection} if each component of their intersection is an S-to-R path. This definition can be generalised to DSR graphs in a natural way, but to avoid technicalities regarding cycle orientation, the reader is referred to \cite{banajicraciun2} for the details. Further notation will be presented as needed.

Returning to the family of CRNs in (\ref{eqfamily}), these give SR graphs shown in Figure~\ref{SRsequence}. If all reactants can influence the rates of reactions in which they participate, then these are also their DSR graphs (otherwise some edges may become directed). Each SR graph contains a single cycle, which is an e-cycle (resp. o-cycle) if $n$ is odd (resp. even). These cycles all fail to be s-cycles because of the unique edge-label of $2$.

\begin{figure}[h]
\begin{minipage}{\textwidth}
\begin{minipage}{0.3\textwidth}
\begin{tikzpicture}[domain=0:12,scale=0.48]

%\draw[very thin,color=black!40] (-4,-2) grid (4,2);

\node at (0,0) {$\mathbf{SYS\,\,}1$};

\path (90: 2cm) coordinate (A2);
\path (30: 2cm) coordinate (R2);
\path (-30: 2cm) coordinate (A3);
\path (270: 2cm) coordinate (R3);
\path (210: 2cm) coordinate (A1);
\path (150: 2cm) coordinate (R1);

\path (-3.5, 1.5) coordinate (B1);
\path (3.5, 1.5) coordinate (B2);

\draw (A1) circle (6pt);
\draw (A2) circle (6pt);
\draw (A3) circle (6pt);

\fill (R1) circle (4pt);
\fill (R2) circle (4pt);
\fill (R3) circle (4pt);

\draw (B1) circle (6pt);
\draw (B2) circle (6pt);

\path (90-20: 2cm) coordinate (A2end);
\path (-30-20: 2cm) coordinate (A3end);
\path (210-20: 2cm) coordinate (A1end);

\path (90+20: 2cm) coordinate (A2end1);
\path (-30+20: 2cm) coordinate (A3end1);
\path (210+20: 2cm) coordinate (A1end1);

\draw[-, dashed, line width=0.04cm] (A2end)  arc (90-20:90-45:2cm);
\draw[-, dashed, line width=0.04cm] (A3end)  arc (-30-20:-30-45:2cm);
\draw[-, dashed, line width=0.04cm] (A1end)  arc (210-20:210-45:2cm);

\draw[-, dashed, line width=0.04cm] (A2end1)  arc (90+20:90+45:2cm);
\draw[-, dashed, line width=0.04cm] (A3end1)  arc (-30+20:-30+45:2cm);
\draw[-, line width=0.04cm] (A1end1)  arc (210+20:210+45:2cm);

\draw[-, line width=0.04cm] (1.8-4, 1.1) -- (1.1-4, 1.3);
\draw[-, line width=0.04cm] (2.2, 1.1) -- (2.9, 1.3);

\node at (-0.8, -1.2) {$2$};

\end{tikzpicture}
\end{minipage}
\begin{minipage}{0.3\textwidth}

\begin{tikzpicture}[domain=-6:6,scale=0.22]

\path (0,0) coordinate (origin);

\node at (0,0) {$\mathbf{SYS\,\,}2$};

\path (45: 6cm) coordinate (R2);
\path (-45: 6cm) coordinate (R3);
\path (135: 6cm) coordinate (R1);
\path (225: 6cm) coordinate (R4);

\path (90: 6cm) coordinate (A2);

\path (0: 6cm) coordinate (A3);
\path (270: 6cm) coordinate (A4);
\path (180: 6cm) coordinate (A1);

\draw (A1) circle (12pt);
\draw (A2) circle (12pt);
\draw (A3) circle (12pt);
\draw (A4) circle (12pt);

\fill (R1) circle (8pt);
\fill (R2) circle (8pt);
\fill (R3) circle (8pt);
\fill (R4) circle (8pt);

\path (90-12: 6cm) coordinate (A2end);
\path (0-12: 6cm) coordinate (A3end);
\path (270-12: 6cm) coordinate (A4end);
\path (180-12: 6cm) coordinate (A1end);

\path (90+12: 6cm) coordinate (A2end1);
\path (0+12: 6cm) coordinate (A3end1);
\path (270+12: 6cm) coordinate (A4end1);
\path (180+12: 6cm) coordinate (A1end1);

\draw[-, dashed, line width=0.04cm] (A2end) arc (90-12:90-35:6cm);
\draw[-, dashed, line width=0.04cm] (A3end) arc (-12:-35:6cm);
\draw[-, dashed, line width=0.04cm] (A4end) arc (270-12:270-35:6cm);
\draw[-, dashed, line width=0.04cm] (A1end) arc (180-12:180-35:6cm);

\draw[-, dashed, line width=0.04cm] (A2end1) arc (90+12:90+35:6cm);
\draw[-, dashed, line width=0.04cm] (A3end1) arc (12:35:6cm);
\draw[-, dashed, line width=0.04cm] (A4end1) arc (270+12:270+35:6cm);
\draw[-, line width=0.04cm] (A1end1) arc (180+12:180+35:6cm);

\path (-8,5) coordinate (B1);
\path (8,5) coordinate (B2);
\path (8,-5) coordinate (B3);

\draw (B1) circle (12pt);
\draw (B2) circle (12pt);
\draw (B3) circle (12pt);

\draw[-, line width=0.04cm] (-5.1, 4.5) -- (-6.9,4.8);
\draw[-, line width=0.04cm] (5.1, 4.5) -- (6.9,4.8);
\draw[-, line width=0.04cm] (5.1, -4.5) -- (6.9,-4.8);

\node at (-4.6, -2) {$2$};

\end{tikzpicture}
\end{minipage}
\begin{minipage}{0.3\textwidth}

\begin{tikzpicture}[domain=-6:6,scale=0.22]

\path (0,0) coordinate (origin);

\node at (0,0) {$\mathbf{SYS\,\,}3$};

\path (90+72: 6cm) coordinate (A2);
\path (90: 6cm) coordinate (A3);
\path (90+4*72: 6cm) coordinate (A4);
\path (90+3*72: 6cm) coordinate (A5);
\path (90+2*72: 6cm) coordinate (A1);

\path (54+72: 6cm) coordinate (R2);
\path (54: 6cm) coordinate (R3);
\path (54+4*72: 6cm) coordinate (R4);
\path (54+3*72: 6cm) coordinate (R5);
\path (54+2*72: 6cm) coordinate (R1);

\draw (A1) circle (12pt);
\draw (A2) circle (12pt);
\draw (A3) circle (12pt);
\draw (A4) circle (12pt);
\draw (A5) circle (12pt);

\fill (R1) circle (8pt);
\fill (R2) circle (8pt);
\fill (R3) circle (8pt);
\fill (R4) circle (8pt);
\fill (R5) circle (8pt);

\path (90+1*72-11: 6cm) coordinate (A2end);
\path (90-11: 6cm) coordinate (A3end);
\path (90+4*72-11: 6cm) coordinate (A4end);
\path (90+3*72-11: 6cm) coordinate (A5end);
\path (90+2*72-11: 6cm) coordinate (A1end);

\path (90+11*72+11: 6cm) coordinate (A2end1);
\path (90+11: 6cm) coordinate (A3end1);
\path (90+4*72+11: 6cm) coordinate (A4end1);
\path (90+3*72+11: 6cm) coordinate (A5end1);
\path (90+2*72+11: 6cm) coordinate (A1end1);

\draw[-, dashed, line width=0.04cm] (A2end) arc (90+72-11:90+72-27:6cm);
\draw[-, dashed, line width=0.04cm] (A3end) arc (90-11:90-27:6cm);
\draw[-, dashed, line width=0.04cm] (A4end) arc (90-72-11:90-72-27:6cm);
\draw[-, dashed, line width=0.04cm] (A5end) arc (90+3*72-11:90+3*72-27:6cm);
\draw[-, dashed, line width=0.04cm] (A1end) arc (90+2*72-11:90+2*72-27:6cm);

\draw[-, dashed, line width=0.04cm] (A2end1) arc (90+72+11:90+72+27:6cm);
\draw[-, dashed, line width=0.04cm] (A3end1) arc (90+11:90+27:6cm);
\draw[-, dashed, line width=0.04cm] (A4end1) arc (90-72+11:90-72+27:6cm);
\draw[-, dashed, line width=0.04cm] (A5end1) arc (90+3*72+11:90+3*72+27:6cm);
\draw[-, line width=0.04cm] (A1end1) arc (90+2*72+11:90+2*72+27:6cm);

\path (-9, -2.5) coordinate (B1);
\path (-7, 5.5) coordinate (B2);
\path (7, 5.5) coordinate (B3);
\path (9, -2.5) coordinate (B4);

\draw (B1) circle (12pt);
\draw (B2) circle (12pt);
\draw (B3) circle (12pt);
\draw (B4) circle (12pt);

\draw[-, line width=0.04cm] (-4.3, 5) -- (-6,5.3);
\draw[-, line width=0.04cm] (4.3, 5) -- (6,5.3);

\draw[-, line width=0.04cm] (-6.5, -2) -- (-8,-2.3);
\draw[-, line width=0.04cm] (6.5, -2) -- (8,-2.3);

\node at (-1.6, -4.8) {$2$};

\end{tikzpicture}
\end{minipage}
\end{minipage}

\caption{\label{SRsequence} The structure of the SR graphs for {\bf SYS} $1, 2$ and $3$ in (\ref{eqfamily}). For simplicity vertices are unlabelled, but filled circles are S-vertices while open circles are R-vertices. Unlabelled edges have edge-label $1$. }
\end{figure}
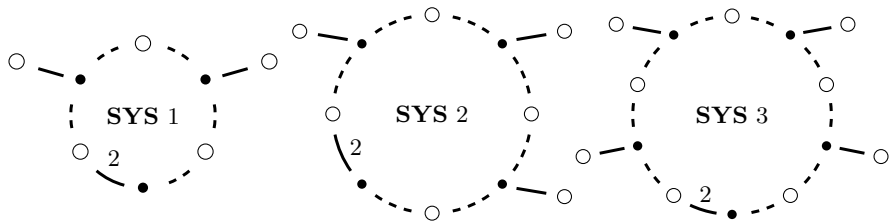

\section{Existing results on CRNs, injectivity and monotonicity}

\subsection{Injectivity and multiple equilibria}
A function $f:X \to \mathbb{R}^n$ is {\bf injective} if for any $x, y \in X$, $f(x) = f(y)$ implies $x = y$. Injectivity of a vector field on some domain is sufficient to guarantee that there can be no more than one equilibrium on this domain. Define the following easily computable condition on an SR or DSR graph:\\

{\bf Condition~($*$)}: All e-cycles are s-cycles, and no two e-cycles have S-to-R intersection.\\

Note that if an SR/DSR graph has no e-cycles, then Condition~($*$) is trivially fulfilled. A key result in \cite{banajicraciun} was: 

\begin{proposition}
\label{inj1}
An N1C reaction system of the form (\ref{basicoutflows}) with SR graph satisfying Condition~($*$) is injective. 
\end{proposition}
\begin{proof}
See Theorem 1 in \cite{banajicraciun}.
\end{proof}

In \cite{banajicraciun2} this result was strengthened considerably and extended beyond CRNs. In the context of CRNs with N1C kinetics it specialises to:

\begin{proposition}
\label{inj2}
An N1C reaction system of the form (\ref{basicoutflows}) with DSR graph satisfying Condition~($*$) is injective. 
\end{proposition}
\begin{proof}
See Corollary 4.2 in \cite{banajicraciun2}.
\end{proof}

Proposition~\ref{inj2} is stronger than Proposition~\ref{inj1} because irreversibility is taken into account. In the case without outflows (\ref{basic}), attention must be restricted to some fixed stoichiometric class. The results then state that no stoichiometry class can contain more than one nondegenerate equilibrium in the interior of the positive orthant \cite{banajicraciun2,craciun2}. (In this context, a degenerate equilibrium is defined to be an equilibrium with a zero eigenvalue and corresponding eigenvector lying in the stoichiometric subspace.) The case with partial outflows was also treated.

\subsection{Monotonicity}

A closed, convex, solid, pointed cone $K \subset \mathbb{R}^n$ is termed a {\bf proper cone} \cite{berman}. The reader is referred to \cite{berman} for basic definitions related to cones. Any proper cone defines a partial order on $\mathbb{R}^n$ as follows: given two points $x, y \in \mathbb{R}^n$:
\begin{enumerate}
\item $x \geq y \Leftrightarrow x - y \in K$;
\item $x > y \Leftrightarrow x \geq y$ and $x \not = y$;
\item $x \gg y \Leftrightarrow x - y \in \mathrm{int}\,K$.
\end{enumerate}
An extremal ray is a one dimensional face of a cone. A proper cone with exactly $n$ extremal rays is termed {\bf simplicial}. Simplicial cones have the feature that unit vectors on the extremal rays can be chosen as basis vectors for a new coordinate system. Consider some linear subspace $\mathcal{A} \subset \mathbb{R}^n$. Then any closed, convex, pointed cone $K \subset \mathcal{A}$ with nonempty interior in $\mathcal{A}$ is termed $\mathcal{A}$-proper. If, further, $K$ has exactly $\mathrm{dim}(\mathcal{A})$ extremal rays, then $K$ is termed $\mathcal{A}$-simplicial. 

Consider some local semiflow $\phi$ defined on $X \subset \mathbb{R}^n$. Assume that there is some linear subspace $\mathcal{A} \subset \mathbb{R}^n$ with a coset $\mathcal{A}^{'}$ with nonempty intersection with $X$, and such that $\phi$ leaves $\mathcal{A}^{'} \cap X$ invariant. Suppose further that there is an $\mathcal{A}$-proper cone $K$ such that for all $x, y \in \mathcal{A}^{'} \cap X$, $x > y \Rightarrow \phi_t(x) > \phi_t(y)$ for all values of $t \geq 0$ such that $\phi_t(x)$ and $\phi_t(y)$ are defined. Then we say that $\left.\phi\right|_{\mathcal{A}^{'} \cap X}$ {\bf preserves $K$}, and that $\left.\phi\right|_{\mathcal{A}^{'} \cap X}$ is {\bf monotone}. If, further, $x > y \Rightarrow \phi_t(x) \gg \phi_t(y)$ for all values of $t > 0$ such that $\phi_t(x)$ and $\phi_t(y)$ are defined, then $\left.\phi\right|_{\mathcal{A}^{'} \cap X}$ is {\bf strongly monotone}. A local semiflow is monotone with respect to the nonnegative orthant if and only if the Jacobian of the vector field has nonnegative off-diagonal elements, in which case the vector field is termed {\bf cooperative}.  

Returning to (\ref{basicCFSTR}), in the case $q=0$, all stoichiometry classes are invariant, while if $q > 0$, there is a globally attracting stoichiometry class. Conditions for monotonicity of $\phi$ restricted to invariant subspaces of $\mathbb{R}^n$ were discussed extensively in \cite{banajidynsys}. Here the immediate aim is to develop graph-theoretic corollaries of one of these results, and to raise some interesting open questions. 

Given a vector $y \in \mathbb{R}^n$, define
\[
\mathcal{Q}_1(y) \equiv \{v \in \mathbb{R}^n\,|\, v_iy_i \geq 0\}.
\]
A matrix $\Gamma$ is {\bf R-sorted} (resp. {\bf S-sorted}) if any two distinct columns (resp. rows) $\Gamma_i$ and $\Gamma_j$ of $\Gamma$ satisfy $\Gamma_i \in \mathcal{Q}_1(-\Gamma_j)$. A matrix $\Gamma^{'}$ is {\bf R-sortable} (resp. {\bf S-sortable}) if there exists a signature matrix $D$ such that $\Gamma \equiv \Gamma^{'}D$ (resp. $\Gamma \equiv D\Gamma^{'}$) is well-defined, and is R-sorted (resp. S-sorted). 

\begin{proposition}
\label{S-simplic}
Consider a system of N1C reactions of the form (\ref{basicCFSTR}) whose stoichiometric matrix $\Gamma$ is R-sortable, and whose reaction vectors $\{\Gamma_k\}$ are linearly independent. Let $\mathcal{S} = \mathrm{Im}(\Gamma)$. Then there is an $\mathcal{S}$-simplicial cone $K$ preserved by the system restricted to any invariant stoichiometry class, such that each reaction vector is collinear with an extremal ray of $K$. 
\end{proposition}

\begin{proof}
This is a specialisation of Corollary~A7 in \cite{banajidynsys}.
\end{proof}

Systems fulfilling the assumptions of Proposition~\ref{S-simplic}, cannot have periodic orbits intersecting the interior of the positive orthant which are stable on their stoichiometry class. In fact, mild additional assumptions ensure strong monotonicity guaranteeing generic convergence of bounded trajectories to equilibria \cite{hirschsmith,halsmith}. 

\section{Graph-theoretic implications of Proposition~\ref{S-simplic}}
\label{secresults}

Some more notation is needed for the results to follow. The {\bf S-degree} ({\bf R-degree}) of an SR graph $G$ is the maximum degree of its S-vertices (R-vertices). Analogous to the terminology for matrices, a subgraph $E$ is {\bf R-sorted} ({\bf S-sorted}) if each R-to-R (S-to-S) path $E_k$ in $E$ satisfies $P(E_k) = 1$. Note that $E$ is R-sorted if and only if each R-to-R path $E_k$ of length $2$ in $E$ satisfies $P(E_k) = 1$.  

An {\bf R-flip} on a SR/DSR graph $G$ is an operation which changes the signs on all edges incident on some R-vertex in $G$. (This is equivalent to exchanging left and right for the chemical reaction associated with the R-vertex). An {\bf R-resigning} is a sequence of R-flips. An {\bf S-flip} and {\bf S-resigning} can be defined similarly. Given a set of R-vertices $\{R_k\}$ in $G$, the closed neighbourhood of $\{R_k\}$ will be denoted $G_{\{R_k\}}$, i.e., $G_{\{R_k\}}$ is the subgraph consisting of $\{R_k\}$ along with all edges incident on vertices of $\{R_k\}$, and all S-vertices adjacent to those in $\{R_k\}$.

\begin{proposition}
\label{mainthm}
Consider a system of N1C reactions of the form (\ref{basicCFSTR}) with stoichiometric matrix $\Gamma$, and whose reaction vectors $\{\Gamma_k\}$ are linearly independent. Define $\mathcal{S} = \mathrm{Im}(\Gamma)$. Associate with the system the SR graph $G$. Suppose that 
\begin{enumerate}
\item $G$ has S-degree $\leq 2$. 
\item All cycles in $G$ are e-cycles.
\end{enumerate}
Then there is an $\mathcal{S}$-simplicial cone $K$ preserved by the system restricted to any invariant stoichiometry class, such that each reaction vector is collinear with an extremal ray of $K$.  
\end{proposition}
The key idea of the proof is simple: if the system satisfies the conditions of Proposition~\ref{mainthm}, then the conditions of Proposition~\ref{S-simplic} are also met. In this case, the extremal vectors of the cone $K$ define a local coordinate system on each stoichiometry class, such that the (restricted) system is cooperative in this coordinate system. This interpretation in terms of recoordinatisation is best illustrated with an example. 

Consider {\bf SYS}~1 from (\ref{eqfamily}) with SR graph shown in Figure~\ref{SRsequence}~{\em left}, which can easily be confirmed to satisfy the conditions of Proposition~\ref{mainthm}. Define the following matrices:
\[
\Gamma  = \left(\begin{array}{rrr}-1&0&2\\-1&-1&0\\0&-1&-1\\1&0&0\\0&1&0\end{array}\right),\qquad T = \left(\begin{array}{rrr}-1&0&2\\-1&1&0\\0&1&-1\\1&0&0\\0&-1&0\end{array}\right),\qquad T^{'} = \left(\begin{array}{rrrrr}1&-2&2&0&0\\1&-1&\,\,\,2&\,\,\,0&\,\,\,0\\1&-1&1&0&0\end{array}\right)
\]
$\Gamma$, the stoichiometric matrix, has rank $3$, and so Proposition~\ref{mainthm} applies. Let $x_1, \ldots, x_5$ be the concentrations of the five substrates involved, $v_1, v_2, v_3$ be the rates of the three reactions, and $v_{ij} \equiv \frac{\partial v_i}{\partial x_j}$. Assuming that the system is N1C means that $V \equiv [v_{ij}]$ has sign structure
\[
\mathrm{sgn}(V) = \left(\begin{array}{ccccc}+&+&0&-&0\\0&+&+&0&-\\-&0&+&0&0\end{array}\right)
\]
where $+$ denotes a nonnegative quantity, and $-$ denotes a nonpositive quantity. Consider now any coordinates $y$ satisfying $x = Ty$. Note that $T$ is a re-signed version of $\Gamma$. Choosing some left inverse for $T$, say $T^{'}$, gives $y_1 = x_1-2x_2+2x_3$, $y_2 = x_1-x_2+2x_3$ and $y_3 = x_1-x_2+x_3$. (The choice of $T^{'}$ is not unique, but this does not affect the argument.) Calculation gives that $J = T^{'}\Gamma VT$ has sign structure
\[
\mathrm{sgn}(J) = \left(\begin{array}{ccc}-&+&+\\+&-&+\\+&+&-\end{array}\right)\,,
\]
i.e., restricting to any invariant stoichiometry class, the dynamical system for the evolution of the quantities $y_1, y_2, y_3$ is cooperative. Further, the evolution of $\{x_i\}$ is uniquely determined by the evolution of $\{y_i\}$ via the equation $x = Ty$.\\

It is time to return to the steps leading to the proof of Proposition~\ref{mainthm}. In Lemmas~\ref{lem5}~and~\ref{lem6} below, $G$ is an SR graph with S-degree $\leq 2$. This implies the following: consider R-vertices $v$, $v^{'}$ and $v^{''}$ such that $v \not = v^{'}$ and $v \not = v^{''}$ ($v^{'} = v^{''}$ is possible). Assume there exist two distinct short paths in $G$, one from $v$ to $v^{'}$ and one from $v$ to $v^{''}$. These paths must be edge disjoint, for otherwise there must be an S-vertex lying on both $A$ and $B$, and hence having degree $\geq 3$.

\begin{lemma}
\label{lem5}
Suppose $G$ is a connected SR graph with S-degree $\leq 2$, and has some connected, R-sorted, subgraph $E$ containing R-vertices $v^{'}$ and $v^{''}$. Assume that there is a path $C_1$ of length $4$ between $v^{'}$ and $v^{''}$ containing an R-vertex not in $E$. Then either $C_1$ is even or $G$ contains an o-cycle. 
\end{lemma}

\begin{proof}
If $v^{'} = v^{''}$, then $C_1$ is not even, then it is itself and e-cycle. Otherwise consider any path $C_2$ connecting $v^{'}$ and $v^{''}$ and lying entirely in $E$. $C_2$ exists since $E$ is connected, and $P(C_2) = 1$ since $E$ is R-sorted. Since $G$ has S-degree $\leq 2$, and $|C_1| = 4$, $C_1$ and $C_2$ share only endpoints, $v^{'}$ and $v^{''}$, and hence together they form a cycle $C$. If $P(C_1) = -1$, then $P(C) = P(C_2)P(C_1) = -1$, and so $C$ is an o-cycle. \qquad \qed
\end{proof}

\begin{lemma}
\label{lem6}
Suppose $G$ is a connected SR graph with S-degree $\leq 2$ which does not contain an o-cycle. Then it can be R-sorted. 
\end{lemma}

\begin{proof}
The result is trivial if $G$ contains a single R-vertex, as it contains no short R-to-R paths. Suppose the result is true for graphs containing $k$ R-vertices. Then it must be true for graphs containing $k+1$ R-vertices. Suppose $G$ contains $k+1$ R-vertices. Enumerate these R-vertices as $R_1, \ldots, R_{k+1}$ in such a way that $G_{-} \equiv G_{\{R_1, \ldots, R_k\}}$ is connected. This is possible since $G$ is connected. 

By the induction hypothesis, $G_{-}$ can be R-sorted. Having R-sorted $G_{-}$, consider $R_{k+1}$. If all short paths between $R_{k+1}$ and R-vertices in $G_{-}$ have the same parity, then either they are all even and $G$ is R-sorted; or they are all odd, and a single R-flip on $R_{k+1}$ R-sorts $G$. (Note that an R-flip on $R_{k+1}$ does not affect the parity of any R-to-R paths in $G_{-}$.) Otherwise there must be two distinct short paths of opposite sign, between $R_{k+1}$ and R-vertices $v^{'}, v^{''} \in G_{-}$ ($v^{'} = v^{''}$ is possible). Since $G$ has S-degree $\leq 2$, these paths must be edge-disjoint, and together form an odd path of length 4 from $v^{'}$ to $R_{k+1}$ to $v^{''}$. By Lemma~\ref{lem5}, $G$ contains an o-cycle. \qquad \qed

\end{proof}

{\bf PROOF of Proposition~\ref{mainthm}.} From Lemma~\ref{lem6}, if no connected component of $G$ contains an o-cycle then each connected component of $G$ (and hence $G$ itself) can be R-sorted. The fact that $G$ can be R-sorted corresponds to choosing a signing of the stoichiometric matrix $\Gamma$ such that any two columns $\Gamma_i$ and $\Gamma_j$ satisfy $\Gamma_i \in \mathcal{Q}_1(-\Gamma_j)$. Thus the conditions of Proposition~\ref{S-simplic} are satisfied. \qquad \qed

\section{Examples illustrating the result and its limitations}
\label{secexamples}

{\bf Example 1: SYS $n$ from Section~\ref{secintro}.} It is easy to confirm that the reactions in {\bf SYS} $n$ have linearly independent reaction vectors for all $n$ . Moreover, as illustrated by Figure~\ref{SRsequence}, the corresponding SR graphs contain a single cycle, which, for odd (even) $n$ is an e-cycle (o-cycle). Thus for even $n$, Proposition~\ref{inj1} and subsequent remarks apply, ruling out the possibility of more than one positive nondegenerate equilibrium for (\ref{basic}) on each stoichiometry class, or in the case with outflows (\ref{basicoutflows}), ruling out multiple equilibria altogether; meanwhile, while for odd $n$, Proposition~\ref{mainthm} can be applied to (\ref{basic}) or (\ref{basicCFSTR}), implying that restricted to any invariant stoichiometry class the system is monotone, and the restricted dynamical system cannot have an attracting periodic orbit intersecting the interior of the nonnegative orthant. \\

{\bf Example 2: Generalised interconversion networks.} Consider the following system of chemical reactions: 

\begin{equation}
\label{eqintercon}
A \rightleftharpoons B, \quad A \rightleftharpoons C , \quad A \rightleftharpoons D,\quad B \rightleftharpoons C 
\end{equation}
with  SR graph shown in Figure~\ref{figintercon}. Formally, such systems have R-degree $\leq 2$ and have SR graphs which are S-sorted. Although Proposition~\ref{mainthm} cannot be applied, such ``interconversion networks'', with the N1C assumption, in fact give rise to cooperative dynamical systems \cite{banajidynsys}, and a variety of different techniques give strong convergence results, both with and without outflows \cite{banajimathchem,angelileenheersontag,banajiangeli}.

\begin{figure}[h]
\begin{center}
\begin{tikzpicture}[domain=0:8,scale=0.6]
    %\draw[very thin,color=black!40] (0,0) grid (10,6);

\node[draw,shape=circle] at (1,4) {$D$};
\node[draw,shape=rectangle] at (4,4) {$R_1$};
\node[draw,shape=circle] at (7,4) {$C$};
\node[draw,shape=rectangle] at (10,4) {$R_4$};

\node[draw,shape=rectangle] at (1,1) {$R_2$};
\node[draw,shape=circle] at (4,1) {$A$};
\node[draw,shape=rectangle] at (7,1) {$R_3$};
\node[draw,shape=circle] at (10,1) {$B$};

\draw[-, thick] (4.8,4) -- (6.2,4);
\draw[-, thick] (7.8,4) -- (9.2,4);

\draw[-, dashed, thick] (1.8,1) -- (3.2,1);
\draw[-, dashed, thick] (4.8,1) -- (6.2,1);
\draw[-, thick] (7.8,1) -- (9.2,1);

%% vertical lines

\draw[-, thick] (1,1.8) -- (1,3.2);
\draw[-, dashed, thick] (4,1.8) -- (4,3.2);
\draw[-, dashed, thick] (10,1.8) -- (10,3.2);

\end{tikzpicture}
\end{center}
\caption{\label{figintercon} The SR graph for reaction system \ref{eqintercon}. All edge labels are $1$ and have been omitted. The system preserves the nonnegative orthant.}
\end{figure}
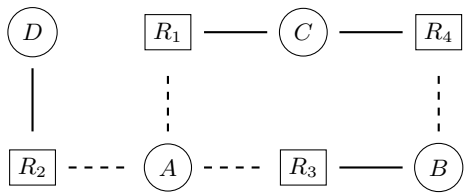

This example highlights that there is an immediate dual to Lemma~\ref{lem6}, and hence Proposition~\ref{mainthm}. The following lemma can be regarded as a restatement of well-known results on systems preserving orthant cones (see \cite{halsmith}, for example, and the discussion for CRNs in \cite{angelileenheersontag}). Its proof is omitted as it follows closely that of Lemma~\ref{lem6}.
\begin{lemma}
\label{S-resign}
Let $G$ be an SR graph with R-degree $\leq 2$ and containing no o-cycles. Then, via an S-resigning, $G$ can be S-sorted. 
\end{lemma}
Although the S-sorting process is formally similar to the R-sorting one, the interpretation of the result is quite different: changing the sign of the $i$th row of $\Gamma$ and the $i$th column of $V$ is equivalent to a recoordinatisation replacing concentration $x_i$ with $-x_i$. Such recoordinatisations give rise to a cooperative system if and only if the original system is monotone with respect to an orthant cone.\\

{\bf Example 3: Linearly independent reaction vectors are not necessary for monotonicity.} Consider the system of three reactions involving four substrates 
\begin{equation}
\label{eqdependent}
A \rightleftharpoons B + C, \qquad B \rightleftharpoons D, \qquad C + D \rightleftharpoons A 
\end{equation}
with stoichiometric matrix $\Gamma$ and SR graph shown in Figure~\ref{SRdependent}.
\begin{figure}[h]
\begin{minipage}{0.3\textwidth}
\[
\Gamma  = \left(\begin{array}{rrr}-1&0&1\\1&-1&0\\1&0&-1\\0&1&-1\end{array}\right)
\]
\end{minipage}
\begin{minipage}{0.6\textwidth}
\begin{center}
\begin{tikzpicture}[domain=0:4,scale=0.6]
    %\draw[very thin,color=black!40] (0,0) grid (10,6);

\node[draw,shape=rectangle] at (1,4) {$R_1$};
\node[draw,shape=circle] at (4,4) {$B$};
\draw[-, thick] (1.8,4) -- (3.2,4);
\node[draw,shape=circle] at (2.5,2.5) {$C$};
\draw[-, thick] (2,3) -- (1.5, 3.5);
\draw[-, dashed, thick] (3,2) -- (3.5, 1.5);
\node[draw,shape=rectangle] at (4,1) {$R_3$};
\node[draw,shape=circle] at (1,1) {$A$};
\node[draw,shape=circle] at (7,1) {$D$};
\node[draw,shape=rectangle] at (7,4) {$R_2$};
\draw[-, thick] (1.7,1) -- (3.3,1);
\draw[-, dashed, thick] (1,1.7) -- (1,3.4);
\draw[-, thick, dashed] (4.7,1) -- (6.3,1);
\draw[-, dashed, thick] (4.7,4) -- (6.3,4);
\draw[-, thick] (7,1.7) -- (7,3.4);
\end{tikzpicture}

\end{center}
\end{minipage}
\caption{\label{SRdependent} The stoichiometric matrix and SR graph for reaction system \ref{eqdependent}. All edge labels are $1$ and have been omitted.}
\end{figure}
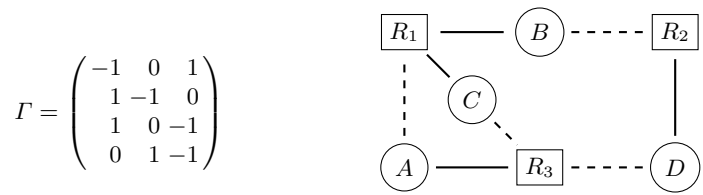

Note that $\Gamma$ is R-sorted, but has rank $2$ as all row-sums are zero. As before, let $x_i$ be the concentrations of the four substrates involved. Now, choose new coordinates $y$ satisfying $x = Ty$, where
\[
T = \left(\begin{array}{rrr}1&\,\,\,0&\,\,\,0\\0&1&0\\-1&0&0\\0&0&1\end{array}\right)\,.
\]
Note: i) $T$ has rank 3, ii) $\mathrm{Im}(\Gamma) \subset \mathrm{Im}(T)$, and iii) regarding the columns of $T$ as extremal vectors of a cone $K$, $K$ has trivial intersection with $\mathrm{Im}(\Gamma)$. One can proceed to choose some left inverse $T^{'}$ of $T$, and calculate that the Jacobian $J = T^{'}\Gamma VT$ has nonnegative off-diagonal entries. In other words the $y$-variables define a cooperative dynamical system. The relationship between $T$ and $\Gamma$ is further discussed in the concluding section. 

Note that although $K$ has empty interior in $\mathbb{R}^4$, both $K$ and $\mathrm{Im}(\Gamma)$ lie in the hyperplane $H=\mathrm{Im}(T)$ defined by $x_1 + x_3 = 0$. As $K$ is $H$-proper, attention can be restricted to invariant cosets of $H$. With mild additional assumptions on the kinetics, the theory in \cite{banajiangeli} can be applied to get strong convergence results, but this is not pursued here. \\

{\bf Example 4a: The absence of o-cycles is not necessary for monotonicity.} Consider the following system of 4 chemical reactions on 5 substrates:

\begin{equation}
\label{eqdependent1}
A \rightleftharpoons B + C, \qquad B \rightleftharpoons D, \qquad C + D \rightleftharpoons A \qquad C + E \rightleftharpoons A
\end{equation}
Define
\[
\Gamma  = \left(\begin{array}{rrrr}-1&0&1&1\\1&-1&0&0\\1&0&-1&-1\\0&1&-1&0\\0&0&0&-1\end{array}\right) \qquad \mbox{and} \qquad T = \left(\begin{array}{rrrr}1&\,\,\,0&\,\,\,0&\,\,\,0\\0&1&0&0\\-1&0&0&0\\0&0&1&0\\0&0&0&1\end{array}\right).
\]
$\Gamma$, the stoichiometric matrix, has rank $3$, and the system has SR graph containing both e- and o-cycles (Figure~\ref{SRdep1}). Further, there are substrates participating in 3 reactions, and reactions involving 3 substrates (and so it is neither R-sortable nor S-sortable). Thus, all the conditions for the results quoted so far in this paper, and for theorems in \cite{angelileenheersontag}, are immediately violated. However, applying theory in \cite{banajidynsys}, the system is order preserving. In particular, $\mathrm{Im}(T)$ is a 4D subspace of $\mathbb{R}^5$ containing $\mathrm{Im}(\Gamma)$ (the stoichiometric subspace), and $T$ defines a cone $K$ which is preserved by the system restricted to cosets of $\mathrm{Im}(T)$.\\

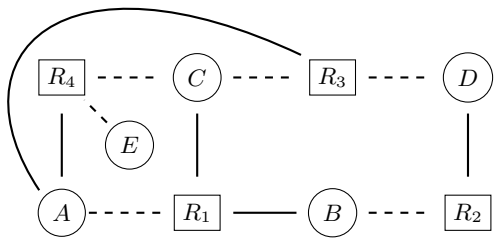
\begin{figure}[h]
\begin{center}
\vspace{-1.5cm}
\begin{tikzpicture}[domain=0:4,scale=0.6]
    %\draw[very thin,color=black!40] (0,0) grid (10,6);

\node[draw,shape=rectangle] at (1,4) {$R_4$};
\node[draw,shape=circle] at (4,4) {$C$};
\draw[-, thick, dashed] (1.8,4) -- (3.2,4);
\node[draw,shape=circle] at (2.5,2.5) {$E$};
\draw[-, thick, dashed] (2,3) -- (1.5, 3.5);

\node[draw,shape=rectangle] at (4,1) {$R_1$};
\node[draw,shape=circle] at (1,1) {$A$};
\node[draw,shape=circle] at (7,1) {$B$};
\node[draw,shape=rectangle] at (10,1) {$R_2$};
\node[draw,shape=rectangle] at (7,4) {$R_3$};
\node[draw,shape=circle] at (10,4) {$D$};
\draw[-, thick, dashed] (1.6,1) -- (3.2,1);
\draw[-, thick] (1,1.8) -- (1,3.2);
\draw[-, thick] (4.8,1) -- (6.2,1);
\draw[-, thick, dashed] (7.8,1) -- (9.2,1);
\draw[-, thick] (10,1.8) -- (10,3.2);
\draw[-, thick, dashed] (7.8,4) -- (9.2,4);
\draw[-, thick, dashed] (4.8,4) -- (6.2,4);
\draw[-, thick] (4,1.8) -- (4,3.2);

\draw[-, thick] (6.3,4.5) .. controls (1,7) and (-1.5,4.5) .. (0.5,1.5);

\end{tikzpicture}

\end{center}
\caption{\label{SRdep1} The SR graph for reaction system \ref{eqdependent1}. All edge labels are $1$ and have been omitted.}
\end{figure}

{\bf Example 4b: The absence of o-cycles is not necessary for monotonicity.} Returning to the system of reactions in (\ref{eqSR1}), the system has SR graph containing an o-cycle (Figure~\ref{figSR1}, {\em left}). Nevertheless, the system was shown in \cite{banajidynsys} to preserve a {\em nonsimplicial} cone for all N1C kinetics. In fact, the further analysis in \cite{banajiangeli} showed that with mild additional assumptions this system is strongly monotone and all orbits on each stoichiometry class converge to an equilibrium which is unique on that stoichiometry class. It is worth mentioning that this example is fundamentally different from Example 4a, and that it is currently unclear how commonly reaction systems preserve orders generated by nonsimplicial cones. 

\section{Discussion and open questions}

The results presented here provide only a glimpse of the possibilities for analysis of limit sets of CRNs using graph-theoretic -- and more generally combinatorial -- approaches. The literature in this area is growing rapidly, and new techniques are constantly being brought into play. Working with the weakest possible kinetic assumptions often gives rise to approaches quite different from those used in the previous study of mass-action systems. Conversely, it is possible that such approaches can be used to provide explicit restrictions on the kinetics for which a system displays some particular behaviour. 

The paper highlights an interesting duality between questions of multistationarity and questions of stable periodic behaviour, a duality already implicit in discussions of interaction graphs \cite{gouze98,soule,kaufman,angelihirschsontag}. Loosely, the absence of e-cycles (positive cycles) is associated with injectivity for systems described by SR graphs (I graphs); and the absence of o-cycles (negative cycles) is associated with absence of periodic attractors for systems described by SR graphs (I graphs). The connections between apparently unrelated SR and I graph results on injectivity have been clarified in \cite{banajiJMAA}, but there is still considerable work to be done to clarify the results on monotonicity.

One open question regards the relationship between the theory and examples presented here on monotonicity, and previous results, particularly Theorem~1 in \cite{angelileenheersontag}, on monotonicity in ``reaction coordinates''. Note that by Proposition~4.5 in \cite{angelileenheersontag} the ``positive loop property'' described there is precisely Conditions~1~and~2 in Proposition~\ref{mainthm} here. At the same time, the requirement that the stoichiometric matrix has full rank, is not needed for monotonicity in reaction coordinates. In some cases (e.g. Example 3 above), it can be shown that this requirement is unnecessary for monotonicity too, but it is currently unclear whether this is always the case. On the other hand, as illustrated by Examples~4a~and~4b, the positive loop property is not needed for monotonicity. 

Consider again Examples~3~and~4a. The key fact is that their stoichiometric matrices admit factorisations $\Gamma = T_1T_2$, taking the particular forms
\[
\left(\begin{array}{rrr}-1&0&1\\1&-1&0\\1&0&-1\\0&1&-1\end{array}\right) = \left(\begin{array}{rrr}1&\,\,\,0&\,\,\,0\\0&1&0\\-1&0&0\\0&0&1\end{array}\right)\left(\begin{array}{rrr}-1&0&1\\1&-1&0\\0&1&-1\end{array}\right)\qquad \mbox{(Example 3), and}
\]
\[
\left(\begin{array}{rrrr}-1&0&1&1\\1&-1&0&0\\1&0&-1&-1\\0&1&-1&0\\0&0&0&-1\end{array}\right) = \left(\begin{array}{rrrr}1&\,\,\,0&\,\,\,0&\,\,\,0\\0&1&0&0\\-1&0&0&0\\0&0&1&0\\0&0&0&1\end{array}\right)\left(\begin{array}{rrrr}-1&0&1&1\\1&-1&0&0\\0&1&-1&0\\0&0&0&-1\end{array}\right)\,\qquad \mbox{(Example 4a).}
\]
In each case, the first factor, $T_1$, has exactly one nonzero entry in each row. On the other hand, the second factor, $T_2$, is S-sorted. The theory in \cite{banajidynsys} ensures that these conditions are sufficient (though not necessary) to guarantee that the system restricted to some coset of $\mathrm{Im}(T_1)$, is monotone with respect to the order defined by $T_1$. The dynamical implications of this factorisation result will be elaborated on in future work. 

A broad open question concerns the extent to which the techniques presented here extend to systems with discrete-time, and perhaps also discrete-state space. In \cite{angelipetrinet}, there were shown to be close relationships, but also subtle differences, between results on persistence in the continuous-time, continuous-state context, and results on liveness in the discrete-time, discrete-state context. Even discretising only time can lead to difficulties: while the interpretation of injectivity results in the context of discrete-time, continuous-state, systems is straightforward, the dynamical implications of monotonicity can differ from the continuous-time case. For example, strongly monotone disrete-time dynamical systems may have stable $k$-cycles for $k \geq 2$ \cite{jiang96JMAA}. When the state space is discrete, an additional difficulty which may arise concerns differentiability of the associated functions, an essential requirement for the results presented here.

Finally, the work on monotonicity here has an interesting relationship with examples presented by Kunze and Siegel, for example in \cite{kunzesiegelpositivity}. This connection remains to be explored and clarified. 

\bibliographystyle{unsrt}

\end{document}